\documentclass[amssymb,12pt,leqno]{amsart}

\usepackage[dvips]{graphics}
\usepackage{amssymb}

\numberwithin{equation}{section}

\newtheorem{thm}{THEOREM}[section]

\newtheorem{cor}[thm]{Corollary}
\newtheorem{prop}[thm]{PROPOSITION}

\newtheorem{quest}[thm]{Question}

 \theoremstyle{definition}

\theoremstyle{remark}
\newtheorem{rem}{Remark}[section]

\newcommand{\tref}[1]{Theorem~\ref{#1}}
\newcommand{\cref}[1]{Corollary~\ref{#1}}
\newcommand{\pref}[1]{Proposition~\ref{#1}}

\newcommand{\R}{\mathbb{R}}

\begin{document}
\pagebreak


\title{On smoothness of isometries between orbit spaces }

\author{Marcos Alexandrino and Alexander Lytchak}



\begin{abstract}
We  discuss the connection between the smooth and metric structure on quotient spaces,
prove smoothness of isometries in special cases and discuss an application to  a conjecture of Molino.

\end{abstract}

\thanks{The first author was  supported by  a CNPq-Brazil research fellowship and partially supported by Fapesp. The second author was supported   by a  Heisenberg grant of the DFG and by the  SFB  878
{\it Groups, geometry and actions}}

\maketitle
\renewcommand{\theequation}{\arabic{section}.\arabic{equation}}
\pagenumbering{arabic}


\section{Smoothness of isometries}

\subsection{The main question}
Given an action of a closed group of isometries $G$ on a  Riemannian manifold $M$,
the quotient $X=M/G$ is equipped with the  natural  quotient metric and a natural quotient ``smooth structure''. The \emph{smooth structure} on $X$ is  given by
the sheaf of ``\emph{smooth functions}"   $\mathcal C ^{\infty} (U) := (\mathcal C^{\infty } (\pi ^{-1} (U)))^G$,
where $U\subset X$ is an arbitrary open set and $\pi:M\to X$ is the canonical projection.
One says that a map $F: M/G \to N /H$  between two quotient spaces is \emph{smooth} if the pull-back by $F$ sends smooth functions on $N/H$  to smooth functions on $M/G$.
If $F$ is bijective and smooth together with its inverse, it is called a diffeomorpism.

 If $X$ is (isometric to) a smooth Riemannian manifold, i.e., if $M \to X$ is a fiber bundle,
then the quotient smooth structure  on $X$ is  the same as the underlying smooth structure of the Riemannian manifold $X$.
In this case, the classical theorem  of Myers-Steenrod  (\cite{msteen}) states that the metric of $X$ determines its smooth structure. This motivates  the following natural and  simple-minded question:

\begin{quest}  \label{firstq} Does the metric of a quotient space $X$ determine its smooth structure?
In other words,   given two manifolds with isometric actions $(M,G)$ and $(N,H)$ and an  isometry
$I:M/G \to N /H$, is $I$  always  a diffeomorphism?
\end{quest}

\subsection{Comments on the smooth structure} \label{subsec2}
Before we proceed, we would like to make a few comments about the smooth structure on quotients.
First and most important: the terms ``smooth structure", ``smooth function", ``diffeomorphism" may be misleading,
since the quotient $X=M/G$ almost never is a smooth manifold.  We adopt the notation of \cite{Schwarz3}
and hope that the terms are not too ambiguous. 

 By definition, the projection $\pi:M \to X=M/G$ is smooth; a map  $F:X\to  Y$ is smooth iff 
$F\circ \pi $ is smooth. Note that the smooth structure (and the metric on the quotient) does not change
if one replaces the action of $G$ by an isometric action of another group, \emph{orbit equivalent} to the action of $G$, i.e., having the same orbits.  The question when a smooth map (a diffeomorphism) $F:M/G \to N/H$  (of a space to itself) can be lifted to a smooth map (diffeomorphism) between the total spaces $M$ and $N$  is highly non-trivial 
(cf. \cite{Strub}, \cite{Schwarz3}, \cite{Schwarz1}).

Clearly, one can always reduce question of smoothness to the the case where the manifolds in question are connected.
Below we will always implicitly  make this connectedness assumption. 
 Using distance functions to orbits, one easily constructs  arbitrary fine smooth partitions of unity in quotient spaces. 
In particular, this implies that all questions concerning smoothness  are local.  

 Let  $p\in M$ be an arbitrary point and let $V_p ^{\perp}$ be the normal space to the fiber through $p$.
The isotropy group $G_p$ acts on $V_p ^{\perp}$  defining the \emph{slice  representation} at $p$.
The famous slice theorem (cf. \cite{Bredon})  says that the 
exponential map $\exp:V_p ^{\perp} \to M$ descends to a local diffeomorphism $O\subset V_p ^{\perp}/G_p \to M/G$ from a small
neighborhood $O$ of $0$ to a small neighborhood $O'$ of $x=\pi (p) \in X$.  The space $V_p ^{\perp}/G_p$ is the tangent space of $X$ at $x$ (in the sense of metric geometry) and the above diffeomorphism will also be denoted by $\exp _x$.

 Let now $I:X=M/G \to Y=N/H$ be an isometry and let $x\in X$ a point with $y=I(x)$. Then $I$ sends geodesics
 starting at $x$ to geodesics starting in $y$. Thus, on a small neighborhood of $x$ we have 
$I= \exp _y  \circ D_xI \circ \exp_x ^{-1}$, for the induced isometry $D_x I$ between the metric tangent spaces
$D_x I: T_x X \to T_y Y$.  Since the exponential maps are local diffeomorphisms, $I$ is a local diffeomorphism at $x$ if and only if the isometry $D_x I$ is a diffeomorphism between the quotient spaces
$T_x X:= V_p ^{\perp} /G_p \to T_y Y =V_q ^{\perp} /H_q$. Here, $p$ and $q$ are arbitrary points in the orbits corresponding to $x$ and $y$ respectively.  This reduces our main Question \ref{firstq} to the case of representations.

Recall that, for a representation of  a compact group $G$ on a vector space $V$, the set of $G$-invariant polynomials is finitely generated by some polynomials $f_1,...,f_n$, due to a  theorem of Hilbert. By a theorem
of Schwarz, any smooth $G$-invariant function on $V$ is a smooth function of the polynomials
 $(f_1,...,f_n)$ (\cite{Schwarz2}). Thus to prove that an isometry between quotients of representation 
is a diffeomorphism, it is sufficient (and necessary) to prove that it preserves invariant polynomials.

The last remark leads to an interesting question in metric geometry. Namely, any $G$-invariant polynomial
of degree $n$ on a representation vector space $V$ of $G$ defines a map on the quotient space $V/G$, whose restriction to each geodesic is polynomial of degree at most $n$.  It seems interesting to recognize  such functions  metrically:

\begin{quest} \label{questalex}
Let $X$ be an Alexandrov space. Let $f$ be a function whose restriction to each geodesic is a polynomial
of degree at most $n$. Does it have some implication on the structure of $X$?
\end{quest}

 For $n=1$, the answer is yes and a splitting result under the assumption that $X$ does not have boundary
has been obtained in  \cite{AB}.  It corresponds to the easy statement that $G$-invariant
linear functions on $V$ can exist only if the set of fixed points of $G$ is non-trivial. 
 Already for $n=2$, Question \ref{questalex} seems to be much harder.

\subsection{Known results}
Probably, the main contribution of this note containing more questions than results, is the observation that Question \ref{firstq}  is non-trivial and interesting for applications and for its own sake.  We are going to explain now, that our Question \ref{firstq} has been answered  in special cases by
some famous theorems.   Beyond the theorem of Myers-Steenrod, mentioned above, there are two important algebraic results. 

The first result is the famous restriction theorem of  Chevalley (\cite{Chev}).
Let $\mathfrak p$ be the tangent space of a symmetric space $M=G/K$ with the 
induced isotropy representation of $K$. Let $\mathfrak a$ be a maximal flat in $\mathfrak p$ and let 
$W$ be the (finite) stabilizer of $\mathfrak a$ in $K$. 
The isotropy representation is \emph{polar}, meaning that 
the embedding $\mathfrak a \to \mathfrak p$ induces an isometry $I:\mathfrak a /W \to \mathfrak p /K$;
we refer to \cite{Alexgor}, \cite{Michor1}, \cite{LT} for basics facts about  polar actions.
The theorem of Chevalley says that this isometry induces an isomorphism between the rings of invariants, i.e., that $I$  is a smooth map.

 The second result is the  theorem of Luna and Richardson (\cite{lr}).  Given an isometric action of
$G$ on $M$ with non-trivial principal isotropy group $H =G_p$ of a regular point  $p\in M$,
 let $\bar M$ be the connected component 
through $p$ of the set of fixed points $M^H$ of $H$.  The normalizer $N(H)$ of $H$ in $G$ stabilizes 
the \emph{generalized section} $\bar M$ (cf. \cite{got}, \cite{Magata}) and the embedding $\bar M \to M$ induces an isometry 
$I:\bar M/N(H) \to M /G$ (to be precise, if $M^H$ is not connected, one may need to replace $N(H)$ by an open 
subgroup of it) . 

 If $M$ is a Euclidean vector space then  the main result of \cite{lr} (after applying a complexification and  using  \cite{Schwarz2}, cf. \cite{str})   says that $I$ is a diffeomorphism.  Applying the reduction procedure to slice representations
from the previous subsection and using that the tangent space of a generalized section is a generalized section of the slice representation,
we deduce that the reduction map $I$ is a diffeomorphism for  arbitrary $M$.

\begin{rem}
In fact, applying Theorem 2.2 from \cite{lr}, instead of the special case we have applied above, 
one deduces that the reduction of $(M,G)$ to any generalized section $(\Sigma, N)$ in the sense of \cite{got}
induces a diffeomorphism $I:\Sigma /N\to M/G$. 
\end{rem}

Finally, we mention the extension of Chevalley's restriction theorem  to general polar actions
by Michor (\cite{Michor1}, \cite{Michor2}).

\subsection{Some new simple observations}

Our first observation is that the analogue of the theorem of Myers-Steenrod is true
if the quotient is isometric to a Riemannian orbifold:

\begin{thm} \label{firstthm}
Let $X= M/G$ be as above. If there is an isometry $I:X \to B$, where $B$ is a smooth Riemannian orbifold, then
$I$ is a diffeomorphism between the quotient smooth structure of $X$ and the underlying smooth orbifold structure of $B$.
\end{thm}

  A few comments before we proceed.  The assumption that $X$ is (isometric to) a Riemannian orbifold  is satisfied in many  geometric
 situations,
  for instance, for actions of cohomogeneity at most $2$, for variationally complete or polar actions (\cite{LT}). The quotient $X$ is a Riemannian orbifold if and only if the action is locally diffeomorphic to a polar action, for which reason, such actions are called \emph{infinitesimally polar}. Another equivalent formulation is that all slice representations are polar.  However, for polar 
actions (representations)  \tref{firstthm} follows from the proof of Michor (\cite{Michor1}, \cite{Michor2}). The reduction of the 
global to the infinitesimal problem discussed in Subsection \ref{subsec2} provides now a proof of \tref{firstthm}. (A slightly different proof of the Theorem will be explained in the last section). 

From the theorem of Myers-Steenrod one deduces that the smooth structure  on a Riemannian orbifold is uniquely determined by the metric.
Hence we get 

\begin{cor} \label{coro}
The answer to Question \ref{firstq} is affirmative if $M/G$ is isometric to a Riemannian orbifold, i.e., if
the action of $G$ on $M$ is infinitesimally polar. 
\end{cor}

We are going to deduce the affirmative answer to our main question in another special case. Again the proof will be an easy
consequence of known results. 
\begin{thm}  \label{secondthm}
Let $X=M/G$ and $Y=N/H$ be of dimension at most $3$. Then each isometry $I:X\to Y$ is a diffeomorphism. 
\end{thm}

\begin{proof}
If the dimension of $X$ is at most $2$, then the action is infinitesimally polar and the result follows from \cref{coro}.
Thus we may assume that $X$ and $Y$ have dimension $3$. Proceeding as in Subsection \ref{subsec2},  we may assume that $M=V$ and $N=W$
are real vector space representations of $G$ and $H$ respectively (we might replace $G$ and $H$ by corresponding isotropy groups).
If one of the representations is polar, so is the other, and the result follows again from  \cref{coro}. Thus we may and will assume that the representations are not polar.  

 We may replace $G$ and $H$ by larger groups $G'\subset O(V)$ and $H'\subset O(V)$ having the same orbits as $G$ and $H$ respectively, whenever it is possible.   Using the theorem of Luna-Richardson, we may replace $(G',V)$ by the generalized section
$(N(G' _p) /G'_p ,V^{G'_p})$ and perform the same operation on $(H',W)$.  
In this way we get an isometry $\hat I :\hat V /\hat G \to \hat W /\hat H$, where $\hat G$ and $\hat H$ act on
$\hat V$ and $\hat W$ respectively, such that the principal isotropy groups are trivial and the groups cannot be enlarged
without enlarging the orbits.  Due to the theorem of Luna-Richardson, all maps  we used producing the ''reduction`` $\hat I$ are 
diffeomorphism. Hence it suffices to prove that $\hat I$ is a diffeomorphism.  

Note that the actions of $\hat G$ (and $\hat H$) are not polar. Now we invoke the classification of all representations of cohomogeneity $3$, as it is discussed in \cite{str}.  From the above assumptions on the representations we deduce that $\hat G$ and $\hat H$ are
one-dimensional and that $\hat V=\hat W=\mathbb R^4$.  In this case, using Section 4 of \cite{str}, one deduces that $\hat G =\hat H$,
and that the representations of $\hat H$ and $\hat G$ on $\mathbb R^ 4$ are equivalent. Hence we may assume that $\hat I$ is an isometry of $\R ^4 /\hat G$ to itself.  In Section 4 of \cite{str} it is shown that any such isometry is induced by an element $J$ in the normalizer of $\hat G$ in $O(4)$. In particular, $\hat I$ preserves the smooth structure.
\end{proof}

\subsection{What to do in general?} The proof of  \tref{firstthm}  as described above is algebraic.
However, at least after the reduction to the case of representations, there is a more geometric proof 
of \tref{firstthm}, essentially contained in  \cite{Terng}, that uses the geometric properties of the Laplacian.
We hope that the use of invariant  differential operators (cf. \cite{mendes}) may help to understand the structure of smooth functions on quotients  and to answer Question \ref{firstq}.

 Another algebraic way,  chosen in the proof of \tref{secondthm}  consists in understanding (classification) 
of representations having isometric quotients.  The description of such \emph{quotient equivalence classes}
of representations is probably possible in each concrete case, but seems to be difficult in general.  At least in small
codimensions, it should be possible to prove the analogue of \tref{secondthm} along the same lines.
Some ideas and results concerning quotient-equivalence classes can be found in  \cite{GL}. These results and the proof
of \tref{secondthm}  motivate the following question. An affirmative answer to it would provide an answer to  Question \ref{firstq}   as well.

\begin{quest}
Let $(V,G)$ and $(W,H)$ be two representations and let $I:V/G \to W/H$ be an isometry. Can $I$ be obtained as a compositions of following three  kinds of  isometries between quotients: The isometry induced by an orbit equivalence
(i.e., replacing a group by a larger group having the same orbit);  the isometry induced by the reduction as in the theorem
of Luna-Richardson;  the isometry $\hat I$ of a quotient $U/K$ to itself induced by an element $g$ in the normalizer of 
$K$ in the orthogonal group $O(U)$?
\end{quest}

\section{Basic forms} 
The quotient space $X=M/G$ contains an open dense Riemannian manifold $X_0$ that consists of the set of
principal orbits. A smooth function on $X$ as defined in the previous section, is just a smooth function on the manifold  $X_0$, whose pull-back to $M$ extends to a smooth function on $M$.  It is natural to define generalizations of smooth forms in the same way. 

A \emph{basic $p$-form} on 
an open subset $U$ of $X$ is a smooth $p$-form on $X_0 \cap U$, whose pull-back to $M$ extends to a smooth $p$-form on 
$\pi ^{-1} (U)$.   The basic forms constitute  a complex (of sheaves on $X$)  and its cohomology coincides with the  singular cohomology of $X$  (The sheaves are fine and the usual lemma of Poincare holds with the usual proof, cf.  \cite{Kosz}).
Despite this fact and the  very natural definition of the complex, the following question seems to be difficult in general:

\begin{quest} \label{secondq} 
Does the smooth structure determine the basic forms on a quotient? In other words,
given a diffeomorphism $F:M_1/G_1 \to M_2/G_2$, does $F$ induce a bijection between basic forms?
\end{quest}     

The following question generalizes Question \ref{firstq}:
\begin{quest} \label{thirdq}
Does an isometry $I:M/G \to N/H$ always preserve the sheaves of basic forms?
\end{quest}

The arguments from Subsection \ref{subsec2} apply to this situation as well and reduce both question to the case of representations.
Again, in the case of infinitesimally polar actions the answer to both question is affirmative and an easy  consequence of known results:

\begin{thm}
If $X=M/G$ is isometric to a smooth Riemannian orbifold then the basic forms of $X$ are precisely the smooth forms of the underlying smooth orbifold.
\end{thm}

\begin{proof}
The reduction explained in Subsection \ref{subsec2} reduces the question to isotropy representations. For infinitesimally
polar actions, the isotropy representations are polar, in which case the result is known (\cite{Michor1}, \cite{Michor2}).
\end{proof}

Thus we deduce:

\begin{cor}
The answer to Question \ref{secondq} and Question \ref{thirdq} is affirmative for infinitesimally polar actions.
\end{cor}

\section{Singular Riemannian foliations} 
\subsection{Definition} The definition  of smooth structures, in particular  of basic forms, on a quotient space 
$X=M/G$  does not depend on the  group action of $G$ on $M$, but only on the decomposition of $M$ into $G$-orbits.
 Such a decomposition is a 
special case of a \emph{singular Riemannian foliation}, a notion that generalizes Riemannian foliations and decomposition in orbits of isometric group actions.  We just recall the definition here and refer the reader to \cite{Molino}, \cite{Alexgor}, \cite{LT}  for more on details.

A \emph{transnormal system} on a Riemannian manifold $M$ is a decomposition of $M$ into pairwise disjoint  isometrically immersed submanifolds $M=\cup_{x} L(x)$, called \emph{leaves}, such that a geodesic starting orthogonally to a leaf remains orthogonal to all leaves it intersects. A\emph{transnormal system} is a called a \emph{singular Riemannian foliation}  if for each leaf $L$ and each $v\in TL$ with footpoint $p,$ there is a vector field $X$ tangent to the leaves so that  $X(p)=v.$ 
While there is some evidence that the answer to the following question is affirmative (see the final remarks of
  \cite{Wilk}), the  question seems to be highly non-trivial:

\begin{quest} \label{trans}
 Is any transnormal system automatically a singular Riemannian foliation?
\end{quest}

\subsection{Smooth structure}
  From now on let $\mathcal F$ be a singular Riemannian foliation on a Riemannian manifold $M$.  A smooth $p$-form
on an open subset $V\subset M$ is called \emph{basic} if the restriction of the form  to the set of regular leaves  $V_{0} \subset V$ is the pull-back of a smooth $p$-form on a local quotient, locally around each point $x\in  V_0$. 
Again, the basic forms constitute a complex of sheaves,  whose cohomology, the \emph{basic cohomology} is 
a very important invariant of the foliation.

 Again locally around each point $p\in M$ the foliation is locally diffeomorph to a (uniquely defined) singular Riemannian
 foliation $T_p \mathcal F$ on the  Euclidean space $T_p M$,  which is invariant with respect to scalar multiplications  (cf. \cite{LT}).
 
   This can be used to reduce questions to the case of singular Riemannian foliations on Euclidean spaces.   However,  most fundamental results known in case of representations (for instance the main theorems from  \cite{Schwarz2} and \cite{Schwarz3}) have not been 
answered in this more general situation until now. We would like to formulate:

\begin{quest}
Let $\mathcal F$ be a  singular Riemannian foliation on a Euclidean space $\mathbb R^n$. Assume that the leaf through
the origin $0$  consists of only one point.  Is the space   of basic functions finitely generated?
\end{quest}

\subsection{Smoothness of isometries}
The singular Riemannian foliation $\mathcal F$ is called \emph{closed} if all of its leaves are closed.
Assume now that $M$ is complete and that $\mathcal F$ is closed. Then the set of leaves $X=M/\mathcal F$
carries a natural quotient metric and the sheaf of basic forms should be considered as a sheaf on $X$.
Again there are arbitrary fine partitions of unity consisting of basic functions and the usual Lemma of Poincare
holds with the usual proof, thus showing  that the basic cohomology coincides  in this case the  singular cohomology
(with real coefficients) of the metric space $X$.

 The following  generalization of  Question \ref{firstq} and Question \ref{thirdq} is further remote from algebra and representation theory, and   one can hope, that the right geometric answer to Question  \ref{firstq} would answer the following question as well. 

\begin{quest} \label{fourthq}
Let $M,N$ be complete Riemannian manifolds  with closed singular Riemannian foliations $\mathcal F$ and $\mathcal G$ respectively.  Does an isometry $I:M/\mathcal F \to N/\mathcal G$ induce a bijection between the  basic forms?  
\end{quest}

While we believe  that the answer to this question is affirmative in general, we only know a proof in a very special situation.
Recall that a singular Riemannian foliation $\mathcal F$ is called \emph{polar}  (also known as a singular Riemannian foliation with section, cf. \cite{Alexandrino}) if   any point $p \in M$ is contained in a submanifold $\Sigma$ (called \emph{section}) that intersects all leaves orthogonally and the regular leaves transversally. 
A singular Riemannian foliation $\mathcal F$ is called \emph{infinitesimally polar} if and only if at all points $p\in M$  
the infinitesimal foliation $T_p \mathcal F$ is polar (\cite{LT}). Again, all singular Riemannian foliations of codimension at most $2$ are infinitesimally polar.
 If $M$ is complete and $\mathcal F$ is closed, infinitesimal polarity  is equivalent to the assumption that $M/\mathcal F$ is isometric to a smooth Riemannian orbifold.  We have:

\begin{thm} \label{sing}
Let $\mathcal F$ be a closed singular Riemannian foliation on a complete Riemannian manifold $M$.
If there is an isometry $I:M/\mathcal F \to B$ to a smooth Riemannian orbifold $B$ then $I$ induces 
an isomorphism between the sheaves of smooth forms on $B$ and the sheaves of basic forms on
$M/\mathcal F$.  In particular, the answer to Question \ref{fourthq} is affirmative for infinitesimally polar foliations. 
\end{thm} 

  A short proof of the theorem above can be obtained along the same lines as the proof of \tref{firstthm} 
relying on \cite{Alexgor}  instead of \cite{Michor1}. We are going to explain  another  only  slightly different proof that is also valid for non-closed foliations.   To do this we recall from \cite{Lyt}
 that for any infinitesimally polar singular Riemannian foliation
$\mathcal F$ on a Riemannian manifold $M$ there is a Riemannian manifold $\hat M$ (called the \emph{geometric resolution}
of $M$) with a regular (!)  Riemannian foliation $\hat {\mathcal F}$ and a smooth surjective map $F:\hat M \to M$,
such that the following holds true.  The preimages of leaves of $\mathcal F$ are leaves of $\hat {\mathcal F}$. The map $F$  is a diffeomorphism, when restricted to $\hat M _0$, the preimage 
of the regular part of $M$. The map $F$ preserves the lengths of all horizontal curves. If $M$ is complete and $\mathcal F$
is closed then so are $\hat M$ and $ \hat {\mathcal F}$,  and $F$ induces an isometry $\bar  F$   between the quotients  
$\hat M / \hat {\mathcal F}$ and $M/\mathcal F$.  Note that since $\hat {\mathcal F}$ is a regular foliation, the $\hat{\mathcal{ F}}$-basic forms 
are the smooth forms on the quotient orbifold $\hat M /\hat {\mathcal F}$ (cf. \cite{Molino}).

Thus the following observation generalizes and proves \tref{sing}.

\begin{prop} \label{help}
Let $\mathcal F$ be an infinitesimally polar singular Riemannian foliation on a Riemannian manifold $M$. 
Let $ F:(\hat M, \hat {\mathcal F}) \to (M,\mathcal F)$ be the geometric resolution of $\mathcal F$.
Then $F$ induces a bijection between the sheaves of basic forms.
\end{prop}

\begin{proof}
Since $F$ maps leaves to leaves, the pull-back by $F$ defines a map from basic forms on $V\subset M$ to
basic forms on $F^{-1} (V) =\hat V$.   Since the restriction to $\hat M _0$ is a diffeomorphism, the pull-back by $F$ is injective. And it remains to prove, that any basic form on $\hat V$ is the pull-back of some form on $V$.
By injectivety, this question is local and it is enough to prove it for a \emph{distinguished neighborhood} $U$
of  a given point $x\in M$.  Identifying the restriction of   $\mathcal F$  to $U$ with a polar foliation and using
the fact that the resolution map  $F$ is constructed in a canonical and local way, we reduce the question to the
case of polar foliations on Euclidean spaces.   Here we apply \cite{Alexgor} to finish the proof. 
\end{proof}

\subsection{Conjecture of Molino}
We would like to discuss an application of Question \ref{fourthq} to the so called Conjecture of Molino.
Given a complete Riemannian manifold $M$, Molino has shown that for any regular Riemannian  foliation
$\mathcal F$ on $M$, the closure $\bar {\mathcal F}$ of $\mathcal F$  consisting of leaf closures of $\mathcal F$
is a singular Riemannian foliation (cf  \cite{Molino}).   Molino has conjectured that  the closure $\mathcal F$
of a \emph{singular} Riemannian foliation is a singular Riemannian foliation as well. 

 For a singular Riemannain foliation $\mathcal F$ on a complete Riemannian manifold $M$ any pair of leaves are equidistant.
Thus so are any pair of closures of two leaves. Therefore, the leaf closure $\bar {\mathcal F}$ is a \emph{transnormal system}.
Thus the problem 
whether $\bar { \mathcal F}$ is a singular Riemannian foliation amounts to   finding smooth vector fields tangent to 
$\bar {\mathcal  F}$ and generating this transnormal system. In particular, the conjecture of Molino
is a special case of Question \ref{trans}.

Using \cite{Schwarz3}, it is possible   to prove that an affirmative answer to Question \ref{firstq}
would prove the conjecture of Molino  for all singular Riemannian foliations, all of whose  infinitesimal foliations are 
given by isometric group actions. Similarly, an affirmative answer to Question \ref{fourthq} together with a generalization
of \cite{Schwarz3} to singular Riemannian foliations would prove the conjecture of Molino in general.  Our interests in the
smoothness issues discussed in this note originated from these observations.  
We would like to finish our exposition by sketching  the proof of the conjecture of Molino for infinitesimally polar foliations,
 a result previously shown in  \cite{ale} for polar foliations.

\begin{thm} \label{lastthm}
Let $ \mathcal F$  be an infinitesimally polar foliation on a complete Riemannian manifold $M$. 
Then the  closure  $\bar {\mathcal F}$ is
a singular Riemannian foliation.
\end{thm}

\begin{proof}
Consider the resolution $\hat M$ of $M$ discussed in the previous subsection.  The map $F: \hat M \to M$
is proper, thus the leaf closures $\hat {\mathcal F}$ are exactly the preimages of the ``leaves" of $\bar {\mathcal F}$.
Due to the theorem of Molino, the closure $\mathcal G$ of $\hat {\mathcal F}$ is a singular Riemannian foliation.
Moreover, locally, the generating smooth vector field of $\mathcal G$ are given by the vector fields generating 
$\hat {\mathcal F}$
and a family of smooth horizontal fields (``the  horizontal lifts of   basic Killing fields", cf. \cite{Molino}).
 However, by duality  with respect to our
Riemannian metric, smooth horizontal fields are in one-to-one correspondence with basic 1-forms.  Due to \pref{help}
these basic one forms descend to basic 1-forms on $M$.  Dualizing again, one obtains   the smooth horizontal vector fields  on $M$ that together with the generating vector fields of $\mathcal F$ generate $\bar {\mathcal F}$.
\end{proof}

\begin{rem}
The above proof shows that the vector fields generating the closure $\bar {\mathcal F}$ in \tref{lastthm}
can locally  be obtained by adding to the vector fields generating $\mathcal F$  the horizontal lifts of ``transversal Killing
fields", as expected by Molino. In the case of infinitesimally polar foliations, these horizontal vector fields are  uniquely determined by their restrictions  to any minimal stratum, cf.  \cite{ale}.
\end{rem}

\noindent\textbf{Acknowledgements} For helpful discussions and comments we would like to express our gratitude to Claudio Gorodski, Gerald Schwarz and Stephan Wiesendorf.

\bibliographystyle{alpha}
\bibliography{smooth}

\end{document}